\documentclass[12pt,a4paper]{article}
\title{Convergence of eigenvalues for a highly non-self-adjoint differential operator}
\author{E. B. Davies
				\footnote{Department of Mathematics,
				King's College London,
				Strand,
				London WC2R 2LS,
				e.brian.davies@kcl.ac.uk}
				\and
				John Weir
				\footnote{Department of Mathematics,
				King's College London,
				Strand,
				London WC2R 2LS,
				john.l.weir@kcl.ac.uk}}
\newtheorem{theorem}{Theorem}[section]
\newtheorem{lemma}[theorem]{Lemma}

\newtheorem{corollary}[theorem]{Corollary}

\newenvironment{proof}[1][Proof]{\begin{trivlist}
\item[\hskip \labelsep {\bfseries #1}]}{\end{trivlist}}
\newenvironment{definition}[1][Definition]{\begin{trivlist}
\item[\hskip \labelsep {\bfseries #1}]}{\end{trivlist}}

\newenvironment{remark}[1][Remark]{\begin{trivlist}
\item[\hskip \labelsep {\bfseries #1}]}{\end{trivlist}}

\newcommand{\qed}{\nobreak \ifvmode \relax \else
      \ifdim\lastskip<1.5em \hskip-\lastskip
      \hskip1.5em plus0em minus0.5em \fi \nobreak
      \vrule height0.75em width0.5em depth0.25em\fi}

\newcommand{\setN}{\mathbf{N}}
\newcommand{\setR}{\mathbf{R}}
\newcommand{\setZ}{\mathbf{Z}}
\newcommand{\setC}{\mathbf{C}}
\newcommand{\eqnref}[1]{(\ref{#1})}
\newcommand{\norm}[1]{\left\| #1\right\|}
\newcommand{\mod}[1]{\left| #1 \right|}
\newcommand{\dom}[1]{\mathrm{Dom}\left(#1\right)}

\newcommand{\conj}[1]{\overline{#1}}
\newcommand{\e}{\mathrm{e}}
\newcommand{\ip}[2]{\left< #1, #2 \right>}

\newcommand{\dd}[1]{\frac{\mathrm{d}}{\mathrm{d} #1}}
\newcommand{\pdd}[1]{\frac{\partial}{\partial #1}}
\newcommand{\spec}[1]{\mathrm{Spec}\left(#1\right)}
\newcommand{\lin}{\mathrm{lin}}

\newcommand{\smooth}{\mathcal{C}^{\infty}([0,1])}

\begin{document}
\maketitle

\begin{abstract}
In this paper we study a family of operators dependent on a small parameter $\varepsilon>0$, which arise in a problem in fluid mechanics. We show that the spectra of these operators converge to $\setN$ as $\varepsilon \to 0$, even though, for fixed $\varepsilon > 0$, the eigenvalue asymptotics are quadratic.

\end{abstract}

\section{Introduction}

In a recent paper \cite{bobs} Benilov, O'Brien and Sazonov argued that the equation
\begin{equation}
	\frac{\partial f}{\partial t} = Hf
	\end{equation}
approximates the evolution of a liquid film inside a rotating horizontal cylinder, where
$H$ is the closure of the operator $H_0$ on $L^2(-\pi,\pi)$ defined by
\begin{equation}\label{eq:h0}
	(H_0f)(\theta) = \varepsilon\pdd{\theta}\left(\sin(\theta)\frac{\partial f}{\partial \theta}\right) + \frac{\partial f}{\partial \theta}
	\end{equation}
for any sufficiently small fixed $\varepsilon > 0$ and all $f \in \dom{H_0} = \mathcal{C}^2_{\mathrm{per}}([-\pi,\pi])$.
They also made several conjectures, based on non-rigorous asymptotic and numerical analysis, including that the spectrum of $H$ is purely imaginary and consists of eigenvalues which accumulate at $\pm i \infty$, and that the eigenvalues converge to $i\setZ$ as $\varepsilon \to 0$. Weir proved these conjectures, except for the convergence of the eigenvalues, in \cite{weir-2007,weir-2008}. We prove the remaining conjecture in this paper.

Davies showed in \cite{davies-2007} that, for $0<\varepsilon \geq 2$, every $\lambda \in \setC$ is an eigenvalue of $H$. For $\varepsilon < 2$ he showed that $-iH$ has compact resolvent by considering the unitarily equivalent operator $A$ on $l^2(\setZ)$ defined by
\begin{equation}
	(Av)_n = \frac{\varepsilon}{2}n(n-1)v_{n-1} - \frac{\varepsilon}{2}n(n+1)v_{n+1}+nv_n
	\end{equation}
for all $v \in \dom{A} = \{v \in l^2(\setZ) : Av \in l^2(\setZ) \}$. Here $A = \mathcal{F}^{-1}(-iH)\mathcal{F}$, where $\mathcal{F}:L^2(-\pi,\pi) \to l^2(\setZ)$ is the Fourier transform. If $\mathcal{F}f=v$ then $(v_n)_{n \in \setZ}$ are the Fourier coefficients of $f$. This result was achieved by obtaining sharp bounds on the rate of decay of eigenvectors and resolvent kernels, and by determining the precise domains of the operators involved.
 He also showed that
\begin{equation}\label{eq:a}
A = A_- \oplus 0 \oplus A_+,
\end{equation}
where $A_-$ and $A_+$ are the restrictions of $A$ to $l^2(\setZ_-)$ and $l^2(\setZ_+)$ respectively, and that $A_-$ is unitarily equivalent to $-A_+$. Since the resolvent is compact and the adjoint has the same eigenvalues, the spectrum of $-iH$ consists entirely of eigenvalues.

Weir proved in \cite{weir-2007} that these eigenvalues, if they exist, must all be real. Boulton, Levitin and Marletta subsequently proved in a recent paper \cite{boulton-2008} that a wider class of operators possess only real eigenvalues. However, they did not prove that any non-zero eigenvalues exist for these operators, nor that their spectra are real. In \cite{weir-2008}, Weir proved rigorously that $-iH$ has infinitely many eigenvalues, all of multiplicity one, which accumulate at $\pm \infty$ by showing that the eigenvalues of $A_+$ correspond to those of a self-adjoint operator with compact resolvent. This operator $L_{\varepsilon}$ on $L^2((0,1),2w_{\varepsilon}(x) \mathrm{d}x)$ is defined as the closure of the operator given by
\begin{equation}\label{eq:leps}
(L_{\varepsilon}f)(x) = - \frac{\varepsilon}{2} w_{\varepsilon}(x)^{-1}(p_{\varepsilon}f')'(x)
\end{equation}
for all $f \in \{f \in \smooth : f(0) = 0\}$,
where
\begin{eqnarray*}
	w_{\varepsilon}(x)	& = & x^{-1} (1-x)^{1/\varepsilon} (1+x)^{-1/\varepsilon},\\
	p_{\varepsilon}(x)	& = & (1-x)^{1+1/\varepsilon} (1+x)^{1-1/\varepsilon}.
\end{eqnarray*}
It was argued in \cite{chugunova-2007} that the distribution of the eigenvalues, if they exist, should be quadratic, but no rigorous bounds were given. By analysing the self-adjoint operator, Weir proved rigorously in \cite{weir-2008} that $\lambda_n \sim \varepsilon \pi^2 n^2 \beta^{-2}$ as $n \to \infty$ for a certain explicit constant $\beta$.

The conjecture in \cite{bobs} that $\lambda_{\varepsilon,n} \to n$ as $\varepsilon \to 0$ is supported by numerical evidence in the same paper, and also in \cite{chugunova-2007, davies-2007, weir-2008}. In this paper we prove the conjecture rigorously. In Section \ref{sec:unitary} we apply unitary transformations to obtain a family of operators on the same space. These operators are invertible with Hilbert-Schmidt inverses (Theorem \ref{thm:unitary}). We then identify a differential
 operator with spectrum $\setN$ (Theorem \ref{thm:cons}, Corollary \ref{cor:sa}), which is unitarily equivalent to an operator whose inverse has an integral kernel which is the pointwise limit of the integral kernels of the aforementioned family of inverses (Theorem \ref{thm:l0m0}). Finally we show that we actually have norm convergence (Theorem \ref{thm:hsc}) and use the variational method to show that this implies convergence of the eigenvalues, i.e. $\lambda_{\varepsilon,n} \to n$ as $\varepsilon \to 0$ for all $n \in \setN$; see Theorem \ref{thm:evc}.

\section{Unitary transformations}\label{sec:unitary}

Since we are interested in the limit as $\varepsilon \to 0$, we assume for the rest of the paper that $0 < \varepsilon < 1$. In order to obtain convergence of operators in some sense, we need a family of operators on the same space. To this end we apply unitary transformations to $L_\varepsilon$ to obtain a family of operators on $L^2((0,1/\varepsilon),\mathrm{d}s)$ and then extend each operator to an operator on $L^2((0,\infty),\mathrm{d}s)$.

\begin{lemma}\label{lem:lepstrans}
The operator $L_{\varepsilon}$, defined by \eqnref{eq:leps}, is unitarily equivalent to an operator $\tilde{L_{\varepsilon}}$ on $L^2((0,1/\varepsilon), \tilde{w}_{\varepsilon}(s)\mathrm{d}s)$ such that
\begin{equation}\label{eq:lepst}
	(\tilde{L_{\varepsilon}}g)(s) = -\tilde{w_{\varepsilon}}(s)^{-1}(\tilde{p}_{\varepsilon}g')'(s)
\end{equation}
where
\begin{eqnarray*}
	\tilde{w_{\varepsilon}}(s)	& = & 2s^{-1} (1-\varepsilon s)^{1/\varepsilon} (1+\varepsilon s)^{-1/\varepsilon}\\
	\tilde{p}_{\varepsilon}(s)	& = & (1-\varepsilon s)^{1+1/\varepsilon} (1+\varepsilon s)^{1-1/\varepsilon}
\end{eqnarray*}
for all $s \in (0,1/\varepsilon)$. Moreover, $\tilde{L}_{\varepsilon}$ is invertible with inverse $R_{\varepsilon}$ given by
\begin{equation}
	(R_{\varepsilon}f)(s) = \int_0^{1/\varepsilon} G_{\varepsilon}(s,t)f(t)\tilde{w_{\varepsilon}}(t) \mathrm{d}t
\end{equation}
for all $f \in L^2((0,1/\varepsilon),\tilde{w_{\varepsilon}}(s)\mathrm{d}s)$, where
\begin{eqnarray}
	\gamma_{\varepsilon}(s) & = & \int_0^s \tilde{p}_{\varepsilon}(u)^{-1} \mathrm{d}u\\
													& = & \frac{1}{2} \left\{\left(\frac{1+\varepsilon s}{1-\varepsilon s}\right)^{1/\varepsilon} - 1  \right\}\label{eq:gameps}
\end{eqnarray}
and
\begin{equation}
	G_{\varepsilon}(s,t) = \left\{ \begin{array}{cc}	 \gamma_{\varepsilon}(s)	&	\textrm{if } 0 < s \leq t < 1/\varepsilon\\
																										 \gamma_{\varepsilon}(t)	&	\textrm{if } 0 < t \leq s < 1/\varepsilon.
									\end{array}\right.
\end{equation}
\end{lemma}
\begin{proof}
We define an operator $U:L^2((0,1),2w_{\varepsilon}(x) \mathrm{d}x) \to L^2((0,1/\varepsilon),\tilde{w}_{\varepsilon}(s)\mathrm{d}s)$ by
\begin{equation}
	(Uf)(s) = f(\varepsilon s)
\end{equation}
for all $f \in L^2((0,1),2w_{\varepsilon}(x) \mathrm{d}x)$, $s \in (0,1/\varepsilon)$. It is easy to show that $U$ is a unitary operator. We define $\tilde{L}_{\varepsilon} = UL_{\varepsilon}U^{-1}$. A simple calculation shows that \eqnref{eq:lepst} holds.
The analogous result about the inverse of $L_{\varepsilon}$ was proven in \cite{weir-2008}, so the result about $\tilde{L}_{\varepsilon}$ follows from the unitary equivalence of the two operators. It only remains for us to calculate
\begin{eqnarray*}
	\gamma_{\varepsilon}(s)	& = & \int_0^s \tilde{p}_{\varepsilon}(u)^{-1} \mathrm{d}u\\
													& = & \frac{1}{\varepsilon}\int_0^{\varepsilon s} p_{\varepsilon}(u)^{-1}\mathrm{d}u\\
													& = & \frac{1}{\varepsilon}\int_0^{\varepsilon s} \frac{1}{1-u^2}\left(\frac{1+u}{1-u}\right)^{1/\varepsilon} \mathrm{d}u\\
													& = & \frac{1}{\varepsilon}\int_0^{\varepsilon s} \frac{1}{(1-u)^2}\left(\frac{1+u}{1-u}\right)^{1/\varepsilon-1} \mathrm{d}u\\
													& = & \frac{1}{2} \left[\left(\frac{1+u}{1-u}\right)^{1/\varepsilon}\right]_0^{\varepsilon s}\\
													& = & \frac{1}{2} \left\{\left(\frac{1+\varepsilon s}{1-\varepsilon s}\right)^{1/\varepsilon} - 1 \right\}.
\end{eqnarray*}
\qed
\end{proof}

\begin{theorem}\label{thm:unitary}
The operator $L_\varepsilon$ is unitarily equivalent to an operator $M_\varepsilon$ on $L^2((0,1/\varepsilon),\mathrm{d}s)$ such that
\begin{equation}
	(M_{\varepsilon}^{-1}f)(s) = \int_0^{1/\varepsilon} K_{\varepsilon}(s,t)f(t)\mathrm{d}t
\end{equation}
for all $f \in L^2((0,1/\varepsilon),\mathrm{d}s)$, where
\begin{equation}\label{eq:keps}
	K_{\varepsilon}(s,t) = \left\{\begin{array}{cc} (st)^{-1/2}\left(\frac{1-\varepsilon s}{1+\varepsilon s}\right)^{1/2\varepsilon}\left\{\left(\frac{1+\varepsilon s}{1-\varepsilon s}\right)^{1/\varepsilon} - 1 \right\}\left(\frac{1-\varepsilon t}{1+\varepsilon t}\right)^{1/2\varepsilon}	& \textrm{if } 0 \leq s \leq t\\
																									 (st)^{-1/2}\left(\frac{1-\varepsilon s}{1+\varepsilon s}\right)^{1/2\varepsilon}\left\{\left(\frac{1+\varepsilon t}{1-\varepsilon t}\right)^{1/\varepsilon} - 1 \right\}\left(\frac{1-\varepsilon t}{1+\varepsilon t}\right)^{1/2\varepsilon}	& \textrm{if } 0 \leq t \leq s.
																 \end{array}\right.
\end{equation}
Moreover, if we define $N_{\varepsilon}$ on $L^2((0,\infty),\mathrm{d}s)$ by
\begin{equation}
	(N_{\varepsilon}f)(s) = \int_0^{\infty} \tilde{K_{\varepsilon}}(s,t)f(t)\mathrm{d}t
\end{equation}
for all $f \in L^2((0,\infty),\mathrm{d}s)$, where
\begin{equation}
	\tilde{K_{\varepsilon}}(s,t) = \left\{	\begin{array}{cc}	 K_{\varepsilon}(s,t)	& \textrm{if } 0 < s,t \leq 1/\varepsilon\\
																														 0											&	\textrm{otherwise},
																					 \end{array}\right.
\end{equation}
then $N_{\varepsilon}$ has the same non-zero eigenvalues as $L_\varepsilon^{-1}$, and each non-zero eigenvalue has the same multiplicity with respect to each operator.
\end{theorem}
\begin{proof}
We define a unitary operator $J_{\varepsilon}:L^2((0,1/\varepsilon),\mathrm{d}s) \to L^2((0,1/\varepsilon),\tilde{w}_{\varepsilon}(s)\mathrm{d}s)$ by
\begin{equation}
(J_{\varepsilon}f)(s) = \tilde{w}_{\varepsilon}(s)^{-1/2}f(s)
\end{equation}
and then put $M_{\varepsilon} = J_{\varepsilon}^{-1}\tilde{L}_{\varepsilon}J_{\varepsilon}$. We then have $M_{\varepsilon}^{-1} = J_{\varepsilon}^{-1}R_{\varepsilon}J_{\varepsilon}$, so
\begin{equation}
	(M_{\varepsilon}^{-1}f)(s) = \int_0^{1/\varepsilon} K_{\varepsilon}(s,t)f(t)\mathrm{d}t
\end{equation}
for all $f \in L^2((0,1/\varepsilon),\mathrm{d}s)$, where
\begin{eqnarray*}
	K_{\varepsilon}(s,t)	& = & \tilde{w}_{\varepsilon}(s)^{1/2}G_{\varepsilon}(s,t)\tilde{w}_{\varepsilon}(t)^{1/2}\\
												& = & \left\{ \begin{array}{cc}	 \tilde{w}_{\varepsilon}(s)^{1/2}\gamma_{\varepsilon}(s)\tilde{w}_{\varepsilon}(t)^{1/2}	 & \textrm{if } 0 < s \leq t < 1/\varepsilon, \\
\tilde{w}_{\varepsilon}(s)^{1/2}\gamma_{\varepsilon}(t)\tilde{w}_{\varepsilon}(t)^{1/2}	 & \textrm{if } 0 < t \leq s < 1/\varepsilon.
																			 \end{array}\right.
\end{eqnarray*}
Substituting in expression \eqnref{eq:gameps} for $\gamma_\varepsilon$, we obtain \eqnref{eq:keps}.

There is no unitary equivalence between $M_{\varepsilon}^{-1}$ and $N_{\varepsilon}$, but it is easy to see that they have the same eigenvalues, and that the non-zero eigenvalues have the same multiplicities with respect to each operator. Indeed, \[N_{\varepsilon} = M_{\varepsilon}^{-1} \oplus 0,\] where $0$ is the zero operator acting on $L^2((1/\varepsilon,\infty), \mathrm{d}s)$. The theorem now follows from the unitary equivalence established in Lemma \ref{lem:lepstrans}.
\qed
\end{proof}

\section{The limit operator}

In this section we consider the operator $L_0$ on $\mathcal{H} = L^2((0,\infty),w_0(s)\mathrm{d}s)$, where
\begin{equation}
w_0(s) = \lim_{\varepsilon\to 0} \tilde{w}_{\varepsilon}(s) = 2s^{-1}\e^{-2s},
\end{equation}
defined on some suitable domain, which we identify below, by
\begin{equation}\label{eq:l0def}
	(L_0f)(s) = -w_0(s)^{-1}(p_0f')'(s),
\end{equation}
where $p_0(s)= \lim_{\varepsilon \to 0} \tilde{p}_{\varepsilon}(s) = \e^{-2s}$. We shall show that $L_0$ has a self-adjoint extension, which is invertible and whose inverse is Hilbert-Schmidt. We then identify the spectrum of $\bar{L_0}$ and the integral kernel of its inverse.

\begin{definition}
Let $\mathcal{P}$ denote the set of all polynomials on $(0,\infty)$ and $s\mathcal{P}$ denote those elements of $\mathcal{P}$ which have constant term zero.
\end{definition}

\begin{definition}
Let $\dom{L_0}$ be the set of twice differentiable functions $g:(0,\infty)\to\setC$ such that $\lim_{s\to 0+}g(s)=0$, $\limsup_{s \to 0+}\mod{g'(s)}<\infty$, $\lim_{s \to \infty}\e^{-s}g(s)=0$, $\lim_{s \to \infty}\e^{-s}g'(s)=0$ and $L_0f \in \mathcal{H}$.
\end{definition}

\begin{remark}
It is easy to show that $\dom{L_0} \subset \mathcal{H}$. We may therefore define $L_0$ on this domain by equation \eqnref{eq:l0def}. Note also that $s\mathcal{P} \subset \dom{L_0}$.
\end{remark}

\begin{lemma}
The set $\mathcal{P}$ is dense in $L^2((0,\infty), 2s\e^{-2s}\mathrm{d}s)$.
\end{lemma}
\begin{proof}
Let $g \in L^2((0,\infty), 2s\e^{-2s}\mathrm{d}s)$ be such that $g \perp \mathcal{P}$ and define
\[ f(s) = g(s) 2s \e^{-2s} \chi_{(0,\infty)}(s) \]
for all $s \in \setR$. We need to prove that $g=0$, or equivalently that $f=0$. We first prove that $f \in L^2(\setR)$ and the Fourier transform $\hat{f}$ of $f$ has an analytic continuation to the set $\{z : \mod{\Im{z}}<1\}$. We then prove that $\hat{f}$ is zero on an interval containing zero, and hence is identically zero. The invertibility of the Fourier transform then implies that $f=0$.

Firstly
\[ \int_{\setR} \mod{f(s)}^2 \mathrm{d}s \leq \sup_{s \in (0,\infty)} 2s\e^{-2s} \int_0^{\infty} \mod{g(s)}^2 2s\e^{-2s}\mathrm{d}s < \infty \]
so $f \in L^2(\setR)$. If $b < 1$ then a similar estimate shows that $\e^{b\mod{s}}f \in L^2(\setR)$. Theorem IX.13 of \cite{reedsimonII} implies that $\hat{f}$ has an analytic continuation to $\{z : \mod{\Im{z}}<1\}$.

In order to evaluate $\hat{f}$ in a neighbourhood of zero, we first approximate $\e^{-i\xi x}$ by polynomials:
\begin{eqnarray*}
	\mod{\e^{-i\xi x} - \sum_{n=0}^{N-1} \frac{(-i\xi x)^n}{n!}}	& \leq & \sum_{n=N}^{\infty} \frac{\mod{x}^n\mod{\xi}^n}{n!}\\
	& = & \mod{x}^N\mod{\xi}^N \sum_{n=0}^{\infty} \frac{\mod{x}^n\mod{\xi}^n}{(n+N)!}\\
	& \leq & \frac{\mod{x}^N\mod{\xi}^N}{N!}\sum_{n=0}^{\infty} \frac{\mod{x}^n\mod{\xi}^n}{n!}\\
	& = & \frac{\mod{x}^N\mod{\xi}^N}{N!} \e^{\mod{x\xi}}
\end{eqnarray*}
for all $x \in \setR$, $\xi \in \setR$, $N \in \setN$. We now choose $\beta \in (0,1)$. If $\delta > 0$ is sufficiently small then $\e^{(\beta+\delta)\mod{x}}f \in L^2(\setR)$. Also $\e^{-\delta\mod{x}} \in L^2(\setR)$, so $\e^{\beta\mod{x}}f \in L^1(\setR)$. If $\mod{\xi} < \beta/2$ then
\begin{eqnarray*}
	\mod{\hat{f}(\xi)} & = & \mod{\int_{\setR}\e^{-i\xi x}f(x)\mathrm{d}x}\\
	& \leq & \mod{\int_{\setR}\left\{\e^{-i\xi x}-\sum_{n=0}^{N-1} \frac{(-i\xi x)^n}{n!}\right\}f(x)\mathrm{d}x} + \mod{\int_{\setR}\sum_{n=0}^{N-1} \frac{(-i\xi x)^n}{n!}f(x)\mathrm{d}x}\\
	& \leq & \int_{\setR} \frac{\mod{x}^N\mod{\xi}^N}{N!} \e^{\mod{x\xi}} \mod{f(x)}\mathrm{d}x + \mod{\int_0^{\infty}\sum_{n=0}^{N-1} \frac{(-i\xi x)^n}{n!}g(x)2x\e^{-2x}\mathrm{d}x}\\
	& \leq & \sup_{x \in (0,\infty)}\frac{\mod{x}^N\mod{\xi}^N}{N!}\e^{(\mod{\xi}-\beta)\mod{x}}\int_{\setR}\mod{f(x)} \e^{\beta \mod{x}}\mathrm{d}x
\end{eqnarray*}
for all $N \in \setN$. The final integral is finite since $\e^{\beta\mod{x}}f \in L^1(\setR)$. We put $h_{\xi,N}(x) = \frac{\mod{x}^N\mod{\xi}^N}{N!}\e^{(\mod{\xi}-\beta)\mod{x}}$ for all $x \in (0,1)$ and note that this is a smooth, positive function of $x$. As $x \to 0$ or $x \to \infty$, $h(x) \to 0$, so the supremum of $h$ must be obtained at a local maximum in the interval $(0,\infty)$. The only zero of $h'$ is at $N/(\beta-\mod{\xi})$ so
\begin{eqnarray*}
	\sup_{x \in (0,\infty)}h(x) & = & h(N/(\beta-\mod{\xi}))\\
	& = & \frac{N^N \e^{-N}}{N!}\frac{\mod{\xi}^N}{(\beta-\mod{\xi})^N}\\
	& \sim & (2\pi N)^{-1/2}\frac{\mod{\xi}^N}{(\beta-\mod{\xi})^N}\\
	& \to & 0
\end{eqnarray*}
as $N \to \infty$ for all $\xi \in [-\beta/2, \beta/2]$. Hence $\hat{f}(\xi) = 0$ for all $\xi \in [-\beta/2, \beta/2]$.
\qed
\end{proof}

\begin{corollary}\label{cor:dense}
The set $s\mathcal{P}$ is dense in $\mathcal{H}$.
\end{corollary}
\begin{proof}
It is easy to show that the operator $U:L^2((0,\infty),2s\e^{-s}\mathrm{d}s) \to \mathcal{H}$ defined by
\[ (Uf)(s) = sf(s) \]
for all $s \in (0,\infty)$ is unitary and that $U\mathcal{P}=s\mathcal{P}$.
\qed
\end{proof}

\begin{theorem}\label{thm:cons}
The operator $L_0$ is symmetric and there is a complete orthonormal sequence $\{e_n\}_{n \in \setN}$ in $\mathcal{H}$ such that, for each $n \in \setN$, $e_n \in \dom{L_0}$ and $L_0e_n = ne_n$.
\end{theorem}
\begin{proof}
For all $f$, $g \in \dom{L_0}$
\begin{eqnarray*}
	\ip{L_0f}{g} &=& -\int_0^{\infty} (p_0f')'(s)\conj{g(s)}\mathrm{d}s\\
	&=&  p_0(0)f'(0)\conj{g(0)} - \lim_{n\to\infty} p_0(n)f'(n)\conj{g(n)} +\int_0^{\infty} p_0(s)f'(s)\conj{g'(s)}\mathrm{d}s\\
	&=& \lim_{n \to \infty} p_0(n)f(n)\conj{g'(n)} - p_0(0)f(0)\conj{g'(0)} -\int_0^{\infty} f(s)\conj{(p_0g')'(s)}\mathrm{d}s\\
	&=& \ip{f}{L_0g}
\end{eqnarray*}
since $f(0)=g(0)=0$, \[\lim_{n \to \infty}p_0(n)f'(n)\conj{g(n)} = \left(\lim_{n \to \infty}\e^{-n}f'(n)\right)\left(\lim_{n \to \infty}\e^{-n}\conj{g(n)}\right) = 0\] and similarly $\lim_{n \to \infty}p_0(n)f(n)\conj{g'(n)} = 0$. Therefore $L_0$ is symmetric.

For all $n \in \setN$, define $a_{n,r}$ recursively for $r=1, \ldots, n$ by
\begin{eqnarray}
	a_{n,n} &=& 1\\
	a_{n,r} &=& -\frac{r(r+1)}{2(n-r)}a_{n,r+1}\textrm{ for }r=1,\ldots,n-1
\end{eqnarray}
and put
\begin{equation}
	f_n(s) = \sum_{r=1}^n a_{n,r}s^r
\end{equation}
for all $s \in (0,\infty)$, $n \in \setN$. Then, for each $n \in \setN$, $f_n \in \dom{L_0}$ and
\begin{eqnarray*}
	(L_0f_n)(s) &=& -\frac{p_0(s)}{w_0(s)}f_n''(s) - \frac{p_0'(s)}{w_0(s)}f'(s)\\
	&=&-\frac{s}{2}f''(s) + sf'(s)\\
	&=&-\sum_{r=2}^n \frac{r(r-1)a_{n,r}}{2}s^{r-1} + \sum_{r=1}^n ra_{n,r}s^r\\
	&=& \sum_{r=1}^{n-1}\left\{ra_{n,r}-\frac{r(r+1)a_{n,r+1}}{2}\right\}s^r +na_{n,n}s^n\\
	&=& nf_n(s)
\end{eqnarray*}
for all $s\in (0,\infty)$. Since the $f_n$ are eigenvectors corresponding to distinct eigenvalues and $L_0$ is symmetric, they are orthogonal. Putting $e_n = f_n\norm{f_n}^{-1}$, we obtain an orthonormal sequence such that $L_0e_n=ne_n$. Each $e_n$ is a polynomial of degree $n$, so $\lin \{e_n\}_{n\in\setN} = s\mathcal{P}$. It now follows from Corollary \ref{cor:dense} that $\{e_n\}_{n\in\setN}$ is complete.
\qed
\end{proof}

The situation above can be seen as an exceptional case of that for the associated Laguerre polynomials $\{L_n^{(\alpha)}\}_{n=0}^{\infty}$ (see \cite{gradshteyn}). If $\alpha>-1$ these polynomials are a complete orthogonal set of eigenvectors, corresponding to eigenvalues $0,1,2,\ldots$ respectively, of the differential equation
\[xf'' + (\alpha + 1 - x)f' + \lambda f = 0\]
on $L^2((0,\infty),x^\alpha\e^{-x}\mathrm{d}x)$. After a change of variables, the eigenvalue equation for $L_0$ becomes the exceptional case $\alpha = -1$. The singularity at the origin requires special treatment, but otherwise the treatment of this case is the same as for $\alpha > -1$.

\begin{corollary}\label{cor:sa}
The operator $L_0$ is essentially self-adjoint, and the spectrum of its closure is precisely $\setN$.
\end{corollary}
\begin{proof}
The result follows immediately from the theorem and Lemma 1.2.2 of \cite{stdo}.
\qed
\end{proof}

\begin{theorem}
The operator $\bar{L_0}$ is invertible and its inverse $R$ is a Hilbert-Schmidt operator. Moreover,
\begin{equation}\label{eq:r0def}
	(R_0f)(s) = \int_0^{\infty} G_0(s,t)f(t)w_0(t)\mathrm{d}t
\end{equation}
for all $f \in \mathcal{H}$, where
\begin{equation}
	G_0(s,t) = \left\{	\begin{array}{cc}	\gamma_0(s) & \textrm{if } 0 < s \leq t\\
																				 \gamma_0(t) & \textrm{if } 0 < t \leq s
											\end{array}\right.
\end{equation}
and
\begin{eqnarray*}
	\gamma_0(s)	& = & \int_0^s p_0(u)^{-1} \mathrm{d}u\\
							& = & \int_0^s \e^{2u} \mathrm{d}u\\
							& = & \frac{1}{2}\left\{\e^{2s}-1\right\}.
\end{eqnarray*}
\end{theorem}
\begin{proof}
Let $R_0$ be defined by equation \eqnref{eq:r0def}. Then $R_0$ is Hilbert-Schmidt since
\begin{eqnarray*}
	\lefteqn{\hspace*{-2em}\int_0^{\infty} \left(\int_0^{\infty} \mod{G_0(s,t)}^2 w_0(t) \mathrm{d}t\right) w_0(s) \mathrm{d}s}\\
	&=& 2\int_0^{\infty} \left(\int_s^{\infty} \mod{G_0(s,t)}^2 w_0(t) \mathrm{d}t\right) w_0(s) \mathrm{d}s\\
	&=& 2\int_0^{\infty} \left(\int_s^{\infty} \mod{\e^{2s}-1}^2t^{-1}\e^{-2t}\mathrm{d}t\right) s^{-1}\e^{-2s}\mathrm{d}s\\
	&\leq& 2\int_0^{\infty} \left(\int_s^{\infty} \e^{-2t}\mathrm{d}t\right)\mod{\e^{2s}-1}^2 s^{-2}\e^{-2s}\mathrm{d}s\\
	&=&\int_0^{\infty} \mod{\e^{-2s}-1}^2 s^{-2}\mathrm{d}s\\
	&\leq& \int_0^1 4 \mathrm{d}s + \int_1^{\infty} s^{-2} \mathrm{d}s\\
	&=& 5.
\end{eqnarray*}

Let $f \in s\mathcal{P}$. Then $f \in L^1((0,\infty),w_0(s)\mathrm{d}s)$. Hence it is easy to show that $Rf$ is absolutely continuous and
\begin{eqnarray*}
	(Rf)'(s) &=& \int_s^\infty \gamma'(s)f(t)w_0(t)\mathrm{d}t\\
	&=& p_0(s)^{-1} \int_s^\infty f(t)w_0(t)\mathrm{d}t.
\end{eqnarray*}
We now show that $Rf \in \dom{L_0}$. Since $f$, $w_0$ and $p_0^{-1}$ are smooth on $(0,\infty)$, $(Rf)$ is twice differentiable (indeed, it is also smooth). We calculate
\begin{equation}
	\limsup_{s\to 0+}\mod{(Rf)(s)} \leq \limsup_{s\to 0+}\gamma_0(s)\int_0^\infty \mod{f(t)}w_0(t)\mathrm{d}t = 0
\end{equation}
and
\begin{eqnarray*}
	\limsup_{s\to 0+}\mod{(Rf)'(s)} &\leq& \limsup_{s\to 0+}p_0(s)^{-1}\int_s^\infty \mod{f(t)}w_0(t)\mathrm{d}t\\
	&=& \int_0^\infty \mod{f(t)}w_0(t)\mathrm{d}t < \infty.
\end{eqnarray*}
Since $f$ is a polynomial, there exists a constant $c>0$ such that $\mod{f(s)}\leq c\e^s$ for all $s \in (0,\infty)$. Let $\delta>0$ be given. Then, for all $s>c/\delta$,
\begin{eqnarray*}
	\mod{\e^{-s}(Rf)(s)} &\leq& \e^{-s}\int_0^s\mod{\gamma_0(t)f(t)w_0(t)}\mathrm{d}t + \e^{-s}\int_s^\infty \mod{\gamma_0(s)f(t)w_0(t)}\mathrm{d}t\\
	&\leq& \e^{-s}\int_0^s\e^{2t}\mod{f(t)}t^{-1}\e^{-2t}\mathrm{d}t + \e^{-s}\int_s^\infty \e^{2s}\mod{f(t)}t^{-1}\e^{-2t}\mathrm{d}t\\
	&\leq& \e^{-s}\int_0^{c/\delta} \mod{f(t)}t^{-1}\mathrm{d}t +\e^{-s}\int_{c/\delta}^s \delta\e^t\mathrm{d}t + s^{-1}\e^s\int_s^\infty c\e^{-t}\mathrm{d}t\\
	&\leq& \e^{-s}\int_0^{c/\delta} \mod{f(t)}t^{-1}\mathrm{d}t + \delta +cs^{-1} \to \delta
\end{eqnarray*}
as $s \to \infty$. Since $\delta>0$ is arbitrary, $\e^{-s}(Rf)(s) \to 0$ as $s \to \infty$. Also
\begin{eqnarray*}
	\mod{\e^{-s}(Rf)'(s)} &\leq& \e^s\int_s^\infty\mod{f(t)}\frac{2}{t}\e^{-2t}\mathrm{d}t\\
	&\leq& 2s^{-1}\e^s\int_s^\infty c\e^{-t}\mathrm{d}t\\
	&=& 2cs^{-1} \to 0
\end{eqnarray*}
as $s \to \infty$. Therefore $Rf \in \dom{L_0}$, provided $L_0Rf \in \mathcal{H}$, and
\begin{eqnarray*}
	(\bar{L_0}Rf)(s) &=& (L_0Rf)(s)\\
	&=& -w_0(s)^{-1}\dd{s}\int_s^\infty f(t)w_0(t)\mathrm{d}t\\
	&=& f(s).
\end{eqnarray*}
Since $f \in \mathcal{H}$, we indeed have $L_0Rf \in \mathcal{H}$.

We have proven that $\bar{L_0}Rf = f$ for all $f \in s\mathcal{P}$. Let $f \in \mathcal{H}$. By Corollary \ref{cor:dense}, there is a sequence $(f_n)$ in $s\mathcal{P}$ such that $f_n \to f$ as $n \to \infty$. By the above, we have $Rf_n \to Rf$ and $\bar{L_0}Rf_n = f_n \to f$ as $n \to \infty$. Since $\bar{L_0}$ is closed, this implies that $Rf \in \dom{\bar{L_0}}$ and $\bar{L_0}Rf = f$. Finally, $\bar{L_0}$ is injective since $0 \notin \spec{\bar{L_0}}$. Hence $\bar{L_0}R\bar{L_0}f = \bar{L_0}f$ implies $R\bar{L_0}f = f$ for all $f \in \mathcal{H}$.
\qed
\end{proof}

\begin{theorem}\label{thm:l0m0}
The operator $\bar{L_0}$ is unitarily equivalent to an operator $M_0$ on $L^2((0,\infty),\mathrm{d}s)$. Moreover,
\begin{equation}\label{eq:k0ker}
	(M_0^{-1}f)(s) = \int_0^{\infty} K_0(s,t)f(t)\mathrm{d}t
\end{equation}
for all $f \in L^2((0,\infty),\mathrm{d}s)$, where
\begin{eqnarray*}
	K_0(s,t)	& = & w_0(s)^{1/2}G_0(s,t)w_0(t)^{1/2}\\
						& = & \left\{	\begin{array}{cc}	(st)^{-1/2} \e^{-s}\left\{\e^{2s}-1\right\}\e^{-t}	&	\textrm{if } 0 < s \leq t\\
																						 (st)^{-1/2} \e^{-s}\left\{\e^{2t}-1\right\}\e^{-t}	& \textrm{if } 0 < t \leq s.
													\end{array}\right.
\end{eqnarray*}
\end{theorem}
\begin{proof}
We define a unitary operator 
\[
J_0:L^2((0,\infty),\mathrm{d}s) \to L^2((0,\infty),w_0(s)\mathrm{d}s)
\]
by
\begin{equation}
	(J_0f)(s) = w_0(s)^{-1/2}f(s)
\end{equation}
and put $M_0 = J_0^{-1}\bar{L_0}J_0$. Then $M_0^{-1} = J_0^{-1}R_0J_0$ and \eqnref{eq:k0ker} is immediate.
\qed
\end{proof}

\section{Convergence of the eigenvalues}

Clearly $\tilde{K_{\varepsilon}}(s,t) \to K_0(s,t)$ as $\varepsilon \to 0$ for each $s$, $t \in (0,\infty)$. In this section we shall show that $K_{\varepsilon} \to K_0$ in $L^2$-norm as $\varepsilon \to 0$, and hence that $N_{\varepsilon} \to M_0^{-1}$ in Hilbert-Schmidt norm as $\varepsilon \to 0$. We use this fact to prove that the eigenvalues of $L_\varepsilon$ converge to those of $L_0$ as $\varepsilon \to 0$.

\begin{lemma}\label{lem:bound}
If $0 \leq s \leq t < 1/\varepsilon$ then
\begin{equation}
	\mbox{$\left(\frac{1-\varepsilon s}{1+\varepsilon s}\right)^{1/2\varepsilon}\left\{\left(\frac{1+\varepsilon s}{1-\varepsilon s}\right)^{1/\varepsilon} - 1 \right\}\left(\frac{1-\varepsilon t}{1+\varepsilon t}\right)^{1/2\varepsilon} - \e^{-s}\left\{\e^{2s}-1\right\}\e^{-t}$} \leq \e^{-s-t}.
\end{equation}
\end{lemma}
\begin{proof}
We first note that
\begin{eqnarray*}
	\mbox{$\log\left(\left(\frac{1+\varepsilon s}{1-\varepsilon s}\right)^{1/2\varepsilon}\left(\frac{1-\varepsilon t}{1+\varepsilon t}\right)^{1/2\varepsilon}\right)$}	& = & \frac{1}{2\varepsilon}\left\{\log(1+\varepsilon s) -\log(1-\varepsilon s)\right\}\\
																	&   &  + \frac{1}{2\varepsilon} \left\{\log(1-\varepsilon t) - \log(1+ \varepsilon t)\right\}\\
																	& = & s - t + \sum_{k=1}^{\infty} \frac{\varepsilon^{2k}}{2k+1}(s^{2k+1}-t^{2k+1})\\
																	& \leq & s-t
\end{eqnarray*}
and hence
\begin{equation}
	\left(\frac{1+\varepsilon s}{1-\varepsilon s}\right)^{1/2\varepsilon}\left(\frac{1-\varepsilon t}{1+\varepsilon t}\right)^{1/2\varepsilon} \leq \e^{s-t}
\end{equation}
for $0 \leq s \leq t < 1/\varepsilon$. For such $s$, $t$,
\begin{eqnarray*}
	\mbox{$\left(\frac{1-\varepsilon s}{1+\varepsilon s}\right)^{1/2\varepsilon}\left\{\left(\frac{1+\varepsilon s}{1-\varepsilon s}\right)^{1/\varepsilon} - 1 \right\}\left(\frac{1-\varepsilon t}{1+\varepsilon t}\right)^{1/2\varepsilon}$}
 & \leq & \mbox{$\left(\frac{1+\varepsilon s}{1-\varepsilon s}\right)^{1/2\varepsilon}\left(\frac{1-\varepsilon t}{1+\varepsilon t}\right)^{1/2\varepsilon}$}\\
	& \leq & \e^{s-t}\\
	& = & \e^{-s}\left\{\e^{2s}-1\right\}\e^{-t} + \e^{-s-t}.
\end{eqnarray*}
\qed
\end{proof}

\begin{theorem}\label{thm:hsc}
$\lim_{\varepsilon \to 0}\norm{N_{\varepsilon} - M_0^{-1}}_{\mathrm{HS}}=0$.
\end{theorem}
\begin{proof}
Using the symmetry of $\tilde{K_{\varepsilon}}$ and $K_0$, it is sufficient to show that
\begin{equation}
\int_0^{\infty}\int_s^{\infty} \mod{\tilde{K_{\varepsilon}}-K_0}^2 \mathrm{d}t\mathrm{d}s \to 0
\end{equation}
as $\varepsilon \to 0$.
If $\frac{1}{2}\log2 \leq s \leq t < 1/\varepsilon$ then, by Lemma \ref{lem:bound},
\begin{eqnarray*}
	st\mod{\tilde{K_{\varepsilon}}-K_0}^2
	& = &  \mbox{$\mod{\left(\frac{1-\varepsilon s}{1+\varepsilon s}\right)^{1/2\varepsilon}\left\{\left(\frac{1+\varepsilon s}{1-\varepsilon s}\right)^{1/\varepsilon} - 1 \right\}\left(\frac{1-\varepsilon t}{1+\varepsilon t}\right)^{1/2\varepsilon} - \e^{-s}\left\{\e^{2s}-1\right\}\e^{-t}}^2$}\\
																			& \leq &  \max\left\{\e^{-2s-2t},\e^{-2s}\left\{\e^{2s}-1\right\}^2\e^{-2t}\right\}\\
																			& \leq & \e^{-2s}\left\{\e^{2s}-1\right\}^2\e^{-2t}.
\end{eqnarray*}
If $t \geq 1/\varepsilon$ then $\tilde{K_{\varepsilon}} = 0$, so the above bound still holds.
Since
\begin{eqnarray*}
	\lefteqn{\hspace*{-3em}\int_{\frac{1}{2}\log2}^{\infty}\int_s^{\infty} (st)^{-1}\e^{-2s}\left\{\e^{2s}-1\right\}^2\e^{-2t} \mathrm{d}t\mathrm{d}s}\\
 & \leq & \int_{\frac{1}{2}\log2}^{\infty} s^{-2}\e^{-2s}\left\{\e^{2s}-1\right\}^2 \int_s^{\infty} \e^{-2t} \mathrm{d}t\mathrm{d}s\\
	& = & \frac{1}{2} \int_{\frac{1}{2}\log2}^{\infty} s^{-2}\e^{-4s}\left\{\e^{2s}-1\right\}^2\mathrm{d}s\\
	& \leq & \frac{1}{2}\int_{\frac{1}{2}\log2}^{\infty}s^{-2}\mathrm{d}s\\
	& < & \infty,
\end{eqnarray*}
we may use the Lebesgue Dominated Convergence Theorem to prove that
\begin{equation}
\int_{\frac{1}{2}\log2}^{\infty}\int_s^{\infty} \mod{\tilde{K_{\varepsilon}}-K_0}^2 \mathrm{d}t\mathrm{d}s \to 0
\end{equation}
as $\varepsilon \to 0$.

If $0\leq s < 1$ and $0<\varepsilon<1$ then
\begin{eqnarray*}
\log\left\{\left(\frac{1+\varepsilon s}{1-\varepsilon s}\right)^{1/\varepsilon}\right\}
& = & 2\sum_{k=0}^{\infty} \frac{\varepsilon^{2k}s^{2k+1}}{2k+1}\\
& \leq & 2\sum_{k=0}^{\infty} \frac{s^{2k+1}}{2k+1}\\
& = & \log\left(\frac{1+s}{1-s}\right),
\end{eqnarray*}
so
\begin{equation}
	\left(\frac{1+\varepsilon s}{1-\varepsilon s}\right)^{1/\varepsilon} \leq \log\left(\frac{1+s}{1-s}\right).
\end{equation}
Also
\begin{equation}
	\left(\frac{1-\varepsilon x}{1+\varepsilon x}\right)^{1/2\varepsilon} \leq \e^{-x}
\end{equation}
if $0 \leq x < 1/\varepsilon$.
Hence, for $0 \leq s \leq \frac{1}{2}\log2$, $s \leq t < 1/\varepsilon$,
\begin{eqnarray*}
	0&\leq &st\mod{\tilde{K_{\varepsilon}}-K_0}^2  \\
 &\leq& \mbox{$\max \left\{\left(\frac{1-\varepsilon s}{1+\varepsilon s}\right)^{1/2\varepsilon}\left\{\left(\frac{1+\varepsilon s}{1-\varepsilon s}\right)^{1/\varepsilon} - 1 \right\}\left(\frac{1-\varepsilon t}{1+\varepsilon t}\right)^{1/2\varepsilon}\!\!, \e^{-s}\left\{\e^{2s}-1\right\}\e^{-t} \right\}^2$}\\
	& \leq & \mbox{$\max \left\{\e^{-2s}\left(\frac{1+s}{1-s} -1\right)^2\e^{-2t}, \e^{-2s}\left\{\e^{2s}-1\right\}^2\e^{-2t} \right\}$}.
\end{eqnarray*}
As before, this bound holds for all $t$ such that $t \geq s$, since $\tilde{K_{\varepsilon}} = 0$ if $t \geq 1/\varepsilon$.
Since we have
\begin{eqnarray*}
	\lefteqn{\hspace*{-3em}\int_0^{\frac{1}{2}\log2} \int_s^{\infty} (st)^{-1} \e^{-2s}\left\{\e^{2s}-1\right\}^2\e^{-2t} \mathrm{d}t\mathrm{d}s}\\
	& \leq & \int_0^{\frac{1}{2}\log2} s^{-2}\e^{-2s}\left\{\e^{2s}-1\right\}^2 \int_s^{\infty} \e^{-2t} \mathrm{d}t\mathrm{d}s\\
	& = & \frac{1}{2} \int_0^{\frac{1}{2}\log2} s^{-2}\e^{-4s}\left\{\e^{2s}-1\right\}^2\mathrm{d}s\\
	& < & \infty
\end{eqnarray*}
and
\begin{eqnarray*}
	\lefteqn{\hspace*{-3em}\int_0^{\frac{1}{2}\log2} \int_s^{\infty} (st)^{-1} \e^{-2s}\left(\frac{1+s}{1-s} -1\right)^2\e^{-2t} \mathrm{d}t\mathrm{d}s}\\
	& = & \int_0^{\frac{1}{2}\log2} \int_s^{\infty} (st)^{-1} \e^{-2s}\left(\frac{2s}{1-s}\right)^2\e^{-2t} \mathrm{d}t\mathrm{d}s\\
	& \leq & c\int_0^{\frac{1}{2}\log2} \int_s^{\infty} \e^{-2s}\e^{-2t} \mathrm{d}t\mathrm{d}s\\
	& = & \frac{c}{2} \int_0^{\frac{1}{2}\log2} \e^{-4s}\mathrm{d}s\\
	& < & \infty
\end{eqnarray*}
for some constant $c$, we may use the Lebesgue Dominated Convergence Theorem to prove that
\begin{equation}
\int_0^{\frac{1}{2}\log2}\int_s^{\infty} \mod{\tilde{K_{\varepsilon}}-K_0}^2 \mathrm{d}t\mathrm{d}s \to 0
\end{equation}
as $\varepsilon \to 0$.
\qed
\end{proof}

We now use standard variational methods to deduce the convergence of the eigenvalues from the norm convergence of the resolvents.

\begin{theorem}\label{thm:evc}
For each $\varepsilon \in (0,1)$, $n \in \setN$, let $\lambda_{\varepsilon,n}$ be the $nth$ eigenvalue of $L_\varepsilon$. Then $\lambda_{\varepsilon,n} \to n$ as $\varepsilon \to 0$.
\end{theorem}
\begin{proof}
If $S$ is a finite-dimensional subspace of $L^2((0,\infty),\mathrm{d}s)$ then we define
\begin{eqnarray}
	\mu_\varepsilon(S) & = & \sup\left\{\ip{(I-N_{\varepsilon})f}{f}:f \in S \textrm{ and } \norm{f}=1\right\},\\
	\mu_0(S) & = & \sup\left\{\ip{(I-S_0^{-1})f}{f}:f \in S \textrm{ and } \norm{f}=1\right\},\\
	\mu_{\varepsilon,n} & = & \inf\{\mu_\varepsilon(S):S \subset L^2((0,\infty),\mathrm{d}s) \textrm{ and }\dim(S)=n\},\\
	\mu_{0,n} & = & \inf\{\mu_0(S):S \subset L^2((0,\infty),\mathrm{d}s) \textrm{ and }\dim(S)=n\}.
\end{eqnarray}
For each $\varepsilon \in (0,1)$, $N_{\varepsilon}$ is self-adjoint since it is a Hilbert-Schmidt operator with a symmetric integral kernel. Also $\spec{L_\varepsilon} \subset [1,\infty)$ and hence it follows from Theorem \ref{thm:unitary} that $I-N_{\varepsilon}$ is a non-negative self-adjoint operator with essential spectrum $\{1\}$ and eigenvalues $\{1-1/\lambda_{\varepsilon,n}:n \in \setN\} \subset [0,1)$. It follows from Corollary \ref{cor:sa} and Theorem \ref{thm:l0m0} that $I-M_0^{-1}$ is a non-negative self-adjoint operator with essential spectrum $\{1\}$ and eigenvalues $\{1-1/n:n\in\setN\} \subset [0,1)$.

By Theorem 4.5.2 of \cite{stdo}, $\mu_{\varepsilon,n} = 1-1/\lambda_{\varepsilon,n}$ and $ \mu_{0,n} = 1-1/n$. Let $\delta > 0$ be given. Then Theorem \ref{thm:hsc} implies that there exists $\eta>0$ such that, whenever $0<\varepsilon<\eta$,
\begin{equation}
	\ip{(I-M_0^{-1})f}{f} - \delta \leq \ip{(I-N_{\varepsilon})f}{f} \leq \ip{(I-M_0^{-1})f}{f} +\delta
\end{equation}
for all $f \in L^2((0,\infty),\mathrm{d}s)$ such that $\norm{f} = 1$. For such $\varepsilon$, this implies that
\begin{equation}
	\mu_\varepsilon(S) - \delta \leq \mu_0(S) \leq \mu_\varepsilon(S) + \delta
\end{equation}
for all finite-dimensional subspaces $S$ of $L^2((0,\infty),\mathrm{d}s)$ and hence that
\begin{equation}
	\mu_{\varepsilon,n} - \delta \leq \mu_{0,n} \leq \mu_{\varepsilon,n} + \delta
\end{equation}
for all $n \in \setN$. Therefore, for all $n \in \setN$, $\mu_{\varepsilon,n} \to \mu_{0,n}$ as $\varepsilon \to 0$ and hence $\lambda_{\varepsilon,n} \to n$ as $\varepsilon \to 0$.
\qed
\end{proof}

%\bibliographystyle{unsrt}
%\bibliography{D:/John/LaTeX/References}

\end{document}